\documentclass[11pt]{amsart}

\usepackage{amssymb,latexsym,amsmath,mathdots,bm,hyperref,marginnote,xcolor,mathabx}
\usepackage{ulem}
\usepackage{tikz-cd}
\usepackage{tikz}
\usetikzlibrary{arrows.meta, decorations.pathmorphing, positioning}

\newtheorem{theorem}{Theorem}[section]
\newtheorem{lemma}[theorem]{Lemma}
\newtheorem{proposition}[theorem]{Proposition}
\newtheorem{corollary}[theorem]{Corollary}

\theoremstyle{definition}
\newtheorem{definition}[theorem]{Definition}

\newtheorem{assumption}[theorem]{Assumption}
\newtheorem{conjecture}[theorem]{Conjecture}

\theoremstyle{remark}
\newtheorem{remark}[theorem]{Remark}

\numberwithin{equation}{section}

\newcommand{\R}{{\mathbb R}}
\newcommand{\Z}{{\mathbb Z}}
\newcommand{\C}{{\mathbb C}}

\newcommand{\G}{{\mathbf G}}

\newcommand{\p}{\mathfrak{p}}
\newcommand{\GL}{{\rm GL}}

\def\k{\mathfrak k}

\begin{document}

\title[$L$-packets and the generic Arthur packet conjectures for $GU(n,n)$]{$L$-packets and the generic Arthur packet conjectures for even unitary similitude groups}


\author{Yeansu Kim}
\address{
Department of Mathematics Education, Chonnam National University, 77 Yongbong-ro, Buk-gu,
Gwangju 61186, South Korea}
\curraddr{}
\email{ykim@jnu.ac.kr}
\thanks{The first author is grateful to Purdue University for providing excellent working conditions during a one-year research visit (July 2022 - July 2023). The first author has been supported by the National Research Foundation of Korea (NRF) grant funded by the Korea government (MSIP) (No. RS-2022-0016551 and  No. RS-2024-00415601 (G-BRL))}

\author{Muthu Krishnamurthy}
\address{}
\curraddr{}
\email{mkrish1939@gmail.com}
\thanks{}

\author{Freydoon Shahidi}
\address{Department of Mathematics
Purdue University
150 N University Street
West Lafayette, IN 47907-2067
United States}
\curraddr{}
\email{shahidi@math.purdue.edu}
\thanks{}

\date{}

\dedicatory{}

\commby{}

\begin{abstract}
We establish the generic local Langlands correspondence by showing the equality of the Langlands-Shahidi $L$-functions and Artin $L$-functions in the case of even unitary similitude groups. As an application, we prove both weak and strong versions of the generic Arthur packet conjectures in the cases of even unitary similitude groups and even unitary groups. We further describe (not necessarily generic) $L$-packets for even unitary similitude groups and establish their expected properties, including Shahidi's conjecture, the finiteness of $L$-packets, and other related results.
\end{abstract}

\maketitle

\section{\bf Introduction}
Let ${\bf G}$ be a connected reductive group defined over a non-archimedean local field $F$ of characteristic zero. Intuitively, the local Langlands correspondence asks whether we can parametrize the set of irreducible admissible representations of ${\bf G}(F)$ by Langlands parameters. Such a set of irreducible admissible representations that correspond to the Langlands parameter is called $L$-packets. One main question in the Langlands program is to study various properties of $L$-packets if it is defined. 

The main purpose of the paper is to prove one main property of $L$-packets, which is a strong version of the generic Arthur packet conjectures in the cases $GU(n,n)$ and $U(n,n)$. Let ${\bf G}_n:=GU(n,n)$ (resp. ${\bf H}_n:=U(n,n)$) be the quasi-split even unitary similitude group (resp. even unitary group) defined in Section \ref{sec:unitary-groups}. Let $\psi$ be an Arthur parameter for ${\bf G}_n(F)$ or ${\bf H}_n(F)$ and $\phi_{\psi}$ be its corresponding Langlands parameter defined in Subsection \ref{s:generic Arthur}. Let $\Pi(\phi_{\psi})$ denote the $L$-packet associated with $\phi_{\psi}$, as defined in the generic case in Definition \ref{Def:generic L-packet} and in the general case in Definition \ref{def_BC}. (See \cite{A82} for more details such as relation between an $L$-packet $\Pi(\pi_{\psi})$ and an Arthur packet $\Pi(\psi)$)
\begin{theorem}[Theorem \ref{StrongConjG} and \ref{StrongConjG:U}]\label{T1:intro}
If the $L$-packet $\Pi(\phi_{\psi})$ for ${\bf G}_n(F)$ (resp. ${\bf H}_n(F)$) has a generic member, then $\Pi(\phi_{\psi})$ is a tempered $L$-packet for ${\bf G}_n(F)$ (resp. ${\bf H}_n(F)$).
\end{theorem}
Note that Theorem \ref{T1:intro}, i.e, a strong version of the generic Arthur packet conjecture can be considered a local version of the generalized Ramanujan conjecture.

Let us briefly explain the main ideas of the proof of Theorem \ref{T1:intro}. The proof consists of three steps. The first step is to show the equality of the Langlands-Shahidi $L$-functions and Artin $L$-functions through the local Langlands correspondence, i.e., the generic local Langlands correspondence for ${\bf G}_n$.

\begin{theorem}[Theorem \ref{Langlands parameter}]\label{T3:intro}
Let $\pi$ be an irreducible admissible generic representation of ${\bf G}_n(F)$. Then, there exists a Langlands parameter $\phi_{\pi}$ such that for any irreducible admissible generic representation $\rho$ of ${\rm GL}(F)$, we have the following:
$$L(s, \rho \times \pi) = L(s, \phi_{\rho} \otimes \phi_{\pi}) \textit{\ \ and \ \ } \gamma(s, \rho \times \pi, \psi_F) = \gamma(s, \phi_{\rho} \otimes \phi_{\pi}, \psi_F),$$
where $\phi_{\rho}$ is the Langlands parameter that corresponds to $\rho$ through the local Langlands correspondence for $GL$ \cite{HT01, H00}.
\end{theorem}

The second step is to prove a weak version of Theorem \ref{T1:intro}—namely, Theorem \ref{T4:intro}—which follows from an application of \cite[Theorem 5.1]{S11} together with the first step. In \cite[Theorem 5.1]{S11}, the third author proved the weak version of the generic Arthur packet conjecture for any (quasi-split) ${\bf G}$, assuming the equality of the Langlands-Shahidi and Artin L-functions for ${\bf G}$. Therefore, our main theorem, i.e., Theorem \ref{T3:intro} implies the following conjecture for ${\bf G}_n$:
\begin{theorem}[Theorem \ref{wConjG:GU}]\label{T4:intro}
If the $L$-packet $\Pi(\phi_{\psi})$ has a generic member, then $\phi_{\psi}$ is a tempered $L$-parameter.   
\end{theorem}

The third step of the proof is to strengthen the above theorem using several properties of $L$-packets. To strengthen it, we first prove that being tempered is preserved through local functorial lift (Lemma \ref{TT}). Note that `if' part of the lemma, that is, if the local functorial lift of an irreducible admissible representation $\pi$ of ${\bf G}_n(F)$ is tempered, then $\pi$ is tempered, is essential in the proof of the theorem. Second, we further prove that tempered property of $L$-packets, i.e., if one member in an $L$-packet is tempered and generic, then all others are tempered (Lemma \ref{lem:tempered $L$-packet}).

Note that in the description of Theorem \ref{T3:intro} and \ref{T4:intro}, we need to specifically describe what $L$-packets mean. The second main purpose of the paper is to describe $L$-packets for $\G_n(F)$ and its corresponding local $L$-functions so that Theorem \ref{T1:intro} becomes completely unconditional. In the case of unitary groups, Mok explicitly describes the Arthur packets, which include the definition of tempered $L$-packets \cite[Theorem 2.5.1(b)]{M15}. We can further define $L$-packets in the general case for ${\bf H}_n(F)$ following the Langlands classification (See \cite{A82} or Section \ref{Removing A}). We use this construction to construct $L$-packets for ${\bf G}_n$(F). More precisely, we first construct the global Base change lifts for ${\bf G}_n(F)$ using the restriction method (Theorem \ref{sbc:non-generic:GU}). With this structure, we define $L$-packets as follows (Definition \ref{Def_L-packets}):

\begin{definition}
Let $\pi_0$ be an irreducible admissible representation of ${\bf G}_n(F)$. Then we define a local $L$-packet that contains $\pi_0$ as a set of irreducible admissible representations $\pi$ of ${\bf G}_n(F)$ such that 
\begin{enumerate}
    \item $\omega_{\pi} = \omega_{\pi_0},$
    \item ${\rm BC}(\pi) \cong {\rm BC}(\pi_0),$ 
\end{enumerate}
where ${\rm BC}(\pi)$ is defined in Definition \ref{def_BC} according to construction in Theorem \ref{sbc:non-generic:GU}.
\end{definition}

\begin{remark}
    The above definition of $L$-packets, together with Definition \ref{def_BC} and Theorem \ref{sbc:non-generic:GU} provide a new definition of Rankin product $L$-functions for ${\bf G}_n(F)$ in the non-generic case (Definition \ref{Lfunction_GU}). 
\end{remark}

A natural question is whether our definition of $L$-packets satisfies expected properties such as Shahidi's conjecture and compatibility with the Langlands classification. Toward the end of the paper, we establish these main properties.

\begin{proposition}[Proposition \ref{property_Lpacket_GU}]
\begin{enumerate}
    \item An $L$-packet for ${\bf G}_n(F)$ consists of finite members.
    \item A tempered $L$-packet contains a generic representation.
    \item $L$-funtions follow the Langlands classification.
    \item $L$-functions defined in Definition \ref{Lfunction_GU} agree with $L$-functions from the Langlands-Shahidi method in the generic case. 
\end{enumerate}    
\end{proposition}

\section{\bf Preliminaries} 
\label{sec:prelims}

\subsection{Unitary Similitude Groups }
\label{sec:unitary-groups}
For any non-archimedean local field ${\mathfrak k}$ of characteristic zero, we write ${\mathfrak o}_{\mathfrak k}$ to denote its ring of integers and $\mathfrak{p}_{\k}$ to denote the unique maximal ideal in ${\mathfrak o}_{\k}$. Let $\varpi_{\k}$ be a uniformizer in $\k$, that is, $\mathfrak{p}_{\mathfrak k} = \varpi_{\k} {\mathfrak o}_{\mathfrak k}$ and let $q_{\mathfrak k} := |{\mathfrak o}_{\mathfrak k}\backslash \mathfrak{p}_{\mathfrak k}|$ be the residual characteristic of ${\mathfrak k}$. We fix the absolute value $|\!|\cdot|\!|_{\mathfrak k}$ that satisfies $|\!|\varpi|\!|_{\mathfrak k}=q_{\mathfrak k}^{-1}$. Now fix $F$,  a non-archimedean local field of characteristic zero. Let $\overline{F}$ be a (separable) algebraic closure of $F$ and let $\Gamma$ denote the Galois group of $\overline{F}/F$. Let $E\subseteq \overline{F}$ be a quadratic extension of $F$ and let $\theta$ denote the corresponding non-trivial Galois automorphism. For $x\in E$, we write $\bar{x}$ to denote $\theta(x)$. We may write $E=F(\beta)$ for some $\beta\in E^{\times}$ satisfying $\bar{\beta}=-\beta$. Let $\psi: F \rightarrow \mathbb{C}$ be a non-trivial additive character of $F$ of level $0$, i.e., $\mbox{ker }\psi\subset {\mathfrak o}_F$ but $\p_F^{-1}\not\subset \mbox{ker }\psi$. Then $\psi_E: E \rightarrow \mathbb{C}$ given by
$\psi_E(x) := \psi({\rm Tr}_{E/F}(x))$ is a non-trivial additive character of $E$ which is trivial on ${\mathfrak o}_E$.
For any algebraic $F$-group $\bf{G}$, we write $G$ to denote  ${\bf G}(F)$, the corresponding group of $F$-points. In addition, given any field $K$ and an embedding $\tau: F\rightarrow K$, we write ${\bf G}\times_{F,\tau}K$ to denote the base change of ${\bf G}$ from $F$ to $K$ through $\tau$. If $K\supset F$ and $\tau=\iota$ is the inclusion map, then we also write ${\bf G}/_{K}$ to denote the group ${\bf G}\times_{F,\iota}K$ obtained by extension of scalars. We do not typeset matrix groups in ``boldface". For example, we write $\text{GL}_n, n\geq 1$, to denote the general linear group, $\text{B}_n$ for the Borel subgroup of upper triangular matrices in $\text{GL}_n$, and $\text{U}_n$ for the unipotent radical of $\text{B}_n$, and write their $F$-points explicitly as $\text{GL}_n(F), \text{B}_n(F)$, and $\text{U}_n(F)$, respectively.
For $n\geq 1$, let $J_n$ be the $n\times n$ matrix 
$J_n:=(a_{i,j})_{i,j}$ with $a_{i,j}=1$ whenever $i+j=n+1$, and $0$ otherwise. Put 
\[
J_{2n}'=\begin{pmatrix}& J_n\\-J_n&\end{pmatrix}.
\]

Let ${\bf G}_n := \text{GU}(n,n)$ be the quasi-split unitary similitude group in $2n$ variables defined with respect to $E/F$ and $J'_{2n}$ \cite{R92}. Here, $\beta J'_{2n}$ must be viewed as the underlying Hermitian form. The group is the $F$-form of $\GL_{2n}\times \GL_1$ determined by the 1-cocycle $(\phi_{\sigma})_{\sigma\in\Gamma}$, where 
\[
\phi_{\sigma}(g,\eta):=\left\{\begin{array}{ll}
(g,\eta)
&
\mbox{ if }\sigma|_{E}=1,\\
(\eta (J'_{2n})^{-1}({}^{t} g^{-1})J'_{2n},\eta)
&
\mbox{ if }\sigma|_{E}=\theta.
\end{array}\right.
\]

The Galois action on $\G_n(\overline{F})=\GL_{2n}(\overline{F})\times \GL_1(\overline{F})$ is the twisted Galois action $(g,\eta)\mapsto \phi_{\sigma}(\sigma(g,\eta))$. For any $F$-algebra ${\mathcal A}$,
\begin{equation}\label{gu}
{{\bf G}_n}({\mathcal A})=\left\{g\in \GL_{2n}(E\otimes_{F} {\mathcal A}): {}^t\bar{g}J'_{2n}g=\eta(g)J'_{2n}, \text{ for some } \eta(g) \in {\mathcal A}^{\times}\right\},
\end{equation}
where $g\mapsto \bar{g}$ denotes the action on $E\otimes_F{\mathcal A}$ obtained from the action of $\theta$ on $E$ by linear extension. This defines a homomorphism $\eta: {\G}_n\longrightarrow {\mathbb G}_m$
of $F$-groups called the {\it similitude} character. The kernel of this homomorphism is the usual (quasi-split) unitary group ${\mathbf H}_n:={\rm U}(n,n)$. We fix the $F$-Borel subgroup {\bf B} of $\G_n$ consisting of upper triangular matrices and write ${\bf B=TU}$, where ${\bf T}$ is a maximal $F$-torus of ${\bf G}_n$ and $\bf U$ is the unipotent radical of ${\bf B}$ consisting of upper triangular unipotent matrices. Their $F$-points are given by 
\[
T = \left\{t=diag(x_1, \cdots, x_n, y_1, \cdots, y_n) | x_i, y_i\in E^{\times}, x_{n-i+1}\overline{y}_i=\eta\in F^{\times}, 1\leq i\leq n,
\right\}
\]
(here $\eta=\eta(t)$) and 
\begin{equation}\label{unip}
U=\left\{n(u,v)=\begin{pmatrix}
u&v\\&\bar{u}^{\star}
\end{pmatrix}; u\in \text{U}_n(E), v\in \text{M}_{n\times n}(E), J_nu^{-1}v={}^t\bar{v}{}^t\overline{u}^{-1}J_n\right\},
\end{equation}
where, for any $n\times n$ matrix, we use $g^{\star}$ to denote $g^{\star}=J_n^{-1}({}^t{g}^{-1})J_n$.
Let ${\bf A}_{0}$ be the maximal $F$-split subtorus of ${\bf T}$, then its $F$-points are given by
$${\rm A}_{0} = \left\{ diag(x_1, \cdots, x_n, y_1, \cdots, y_n)
| x_i ,y_i\in F^{\times}\mbox{ satisfying }x_{n-i+1}y_i=\eta, 1\leq i\leq n\right\}.
$$

The groups ${\mathbf H}_n\subset {\mathbf G}_n$ share the same derived group and there is an exact sequence 
$$
1\longrightarrow {\mathbf H}_n\longrightarrow {\mathbf G}_n \overset{\eta}{\longrightarrow} {\rm GL}_1\longrightarrow 1.
$$
Let $Z=Z_n$ denote the center of the group $G:={\bf G}_n(F)$. It is given by 
$$
diag(z,\cdots,z)\in \GL_{2n}(E), \eta=z\bar{z},
$$
and its intersection $Z_H$ with $H:={\mathbf H}_n(F)$ consists of all $z\in E^{\times}$ satisfying $z\bar{z}=1$. It follows that $\eta$ induces the isomorphism 
\begin{equation}
\label{eqn-quotient}
G/Z H \cong F^{\times}/N_{E/F}(F^{\times}).
\end{equation}

Let $\widetilde{\Sigma}$ denote the set of (absolute) roots of ${\bf G}_n$ with respect to ${\bf T}$. The choice of ${\bf B}$ then determines a set of positive roots $\widetilde{\Sigma}^{+}$ and let $\widetilde{\Delta}$ denote the corresponding set of simple roots. Let $\widetilde{\Sigma}|_{{\bf A}_{0}}=\Sigma$, $\widetilde{\Sigma}^{+}|_{{\bf A}_{0}}=\Sigma^{+}$ and $\widetilde{\Delta}|_{{\bf A}_{0}}=\Delta$. Then $\Sigma$ is the set of roots of ${\bf G}_n$ relative to ${\bf A}_{0}$ and $\Delta$ is the corresponding set of simple roots. The elements of $\Sigma$ are called {{\it relative roots}} of ${\bf G}_n$ (with respect to ${\bf A}_{0}$). (See \cite[\S 21]{B91}.) The group $\Gamma$ acts on $\widetilde{\Sigma}$ and $\tilde{\Delta}$ with $\mbox{Gal}(\overline{F}/E)$ acting trivially. Thus we get an action of $\mbox{Gal}(E/F)$ on $\widetilde{\Sigma}$ and $\widetilde{\Delta}$ partitioning them into finitely many $\mbox{Gal}(E/F)$-orbits; these orbits are in one-to-one correspondence with elements in $\Sigma$ and $\Delta$, respectively. The root system determined by $\widetilde{\Sigma}$ is of type ${}^2A_{2n-1}$ and the root system determined by $\Sigma$ is of type $C_n$. 

For a field extension $F'\supset F$, let us describe ${\bf G}_n(F')$ explicitly. Since $E$ is separable, $E\otimes_{F}F'$ is either (i) a field containing an isomorphic copy of $E$, or (ii) a product (direct) of two fields each containing an isomorphic copy of $E$. In case (i), the group ${\bf G}_n(F')$ is given as in (\ref{gu}). In case (ii), $E\otimes_F F'=F'\times F'$ with associated embeddings $\iota_1,\iota_2$ from $E\longrightarrow F'$, where the $i^{\rm th}$ copy of $F'$ in this product corresponds to $\iota_i$, $i=1,2$. The non-trivial Galois action is then given by swapping two factors, i.e., $(x,y)\mapsto (y,x)$. Consequently
\begin{eqnarray}
\G_n(F')&=&\left\{(g,g')\in \GL_{2n}(F'\oplus F'): {}^tg'J'_{2n}g=\eta J'_{2n}, \eta \in {F'}^{\times}\right\}\nonumber \\
&\cong&\left\{(g,\eta)\in \GL_{2n}(F')\times\GL_1(F')\nonumber\right\}.
\end{eqnarray}
Choosing a distinguished embedding $\iota_1$, we see that ${\bf G}_n\times_{F,\iota_1}{F'}\cong \GL_{2n}\times \GL_1$ via projection onto the first factor. In particular, ${\bf G}_n/_{E}\cong \GL_{2n}\times \GL_{1}$ determined by the identity map $\iota_1$. The maximal torus ${\bf T}/_{E}\subset {\bf G}_n/_{E}$ corresponds to the diagonal torus $(\GL_{1})^{2n}\times \GL_{1}$ under this isomorphism. Thus 
\begin{equation}\label{max-torus}
t\mapsto diag( x_1, x_2,\cdots, x_{2n})
\times\eta\in \GL_{2n}\times \GL_1
\end{equation}
represents a generic element of the maximal torus ${\bf T}/_{E}$. For $1\leq i\leq 2n$, let $e_i$ be the character of ${\bf T}$ that sends $t\mapsto x_i$. Let $e_{0}$ be the character that sends $t\mapsto \eta$. The group $X({\bf T})_E$ of $E$-rational characters of ${\bf T}$ is $\bigoplus_{i=1}^{2n}\Z e_i\oplus\Z e_{0}$ and the non-trivial Galois action is given by $e_{0}\mapsto e_{0}$, $e_{i}\mapsto e_{0}-e_{2n+1-i}$, $1\leq i\leq 2n$. Put $\widetilde{\alpha}_{i}:=e_i-e_{i+1}, 1\leq i\leq 2n-1$, then $\widetilde{\Delta}=\{\widetilde{\alpha}_1,\widetilde{\alpha}_2,\ldots, \widetilde{\alpha}_{2n-1}\}$ whose restriction to ${\rm A}_{0}$ gives $\Delta$, the corresponding set of simple relative roots. 
 
The set of all standard parabolic $F$-subgroups is in bijection with the set of all subsets $\Delta$ \cite[Proposition 21.12]{B91}.  We write $\Delta^{\bf M'}\subset \Delta \leftrightarrow {\bf P}'={\bf M}'{\bf N}'$ to denote this bijection, where ${\bf M}' \supset {\bf T}$ is the corresponding standard $F$-Levi subgroup and ${\bf N}'\subset {\bf U}$ is the unipotent radical of ${\bf P}'$. Here, $\Delta^{{\bf M}'}$ is the analogue of $\Delta$ obtained from replacing ${\bf G}_n$ by ${\bf M}$ in the definitions. The standard Levi $F$-subgroups ${\bf M}'$ are of the following shape: We use $R_{E/F}$ to denote the functor of restriction of scalars, and set ${\rm G}_0:={\rm GU}(0,0)=\GL_1$. Then for $\Delta^{{\bf M}'}\subset \Delta$, there are integers $n_1,\ldots,n_r>0$ and $m\geq 0$ so that $n_1+n_2+\ldots+n_r+m=n$, and 
\[
{\bf M}'\cong R_{E/F}\GL_{n_1}\times R_{E/F}\GL_{n_2}\times\cdots\times R_{E/F}\GL_{n_r}\times \text{\bf G}_m.
\] 
By loc.cit. any parabolic $F$-subgroup of ${\bf G}_n$ is conjugate to one and only one standard ${\bf P}'$ by an element of $G_n$.

\subsection{The Langlands-Shahidi method and relevant $L$-functions}
\label{sec-LS}

We consider the case ${}^2 A_{2n-1}-4$ in \cite{S88} in the context of unitary similitude groups and describe the resulting $L$-functions. (See \cite{KK05} where this case was studied for unitary groups.) Put  $\widetilde{\theta}=\widetilde{\Delta}-\{\widetilde{\alpha}_i,\widetilde{\alpha}_{n+i}\}$, $1\leq i\leq n-1$, and let $\theta\subset \Delta$ be the corresponding subset of simple relative roots. Then $\Delta\backslash \theta=\alpha_i$, where $\alpha_i$ is the restriction of $\widetilde{\alpha}_i$ to ${\bf A}_{0}$. Let ${\bf M}={\bf M}_{\theta}$ be the standard maximal Levi $F$-subgroup of ${\bf G}_n$ associated to $\theta$ under the above mentioned correspondence. Then ${\bf M}\cong R_{E/F}\GL_{k}\times {\bf G}_m$, $k+m=n$, with $k\leq n-1$. Let $\bf P\supset {\bf B}$ be the the maximal $F$-parabolic (standard) subgroup of $\bf G$ with Levi-component $\bf M$ and let $\bf N$ be the corresponding unipotent radical. Then as a subgroup of ${\bf G}_n(F)$,
\[
M=
\left\{m(g,h)=\left[
\begin{array}{c|c|c}
\eta_h g & & \\
\hline
 & h&\\\hline&& \bar{g}^{\star}
\end{array}
\right], g\in \GL_k(E), h\in {\bf G}_m(F)\right\};
\]
and 
 \[
{N} = 
\left\lbrace
\left[\begin{array}{c|c|c}
I_k & X&Y \\\hline 
 & I_{2m} &Z\\\hline &&I_k
\end{array}\right]
| ^t{}\bar{X}(-1)^kJ_{k}+J_{2m}Z=0, J_{k}Y+ ^t{}\bar{Y}(-1)^kJ_{k}= -^t{}\bar{Z}J_{2m}Z
\right\rbrace.
\]
There is a unique Weyl group element $\widetilde{w_{0}}$ (for ${\bf A}_{0}$ in ${\bf G}_n$) satisfying $\widetilde{w_{0}}(\theta)\subset \Delta$ while $\widetilde{w_{0}}(\alpha_n)<0$ which may be represented by the element 
\[
w_{0}:=\left[
\begin{array}{c|c|c}
 & &I_k \\
\hline
 & I_{2m}&\\\hline (-1)^kI_k&& 
\end{array}
\right].
\] 
Observe that $\bf P$ is self-associate, meaning, $w_{0}{\bf M}w_{0}^{-1}={\bf M}$. 

Next, for a field extension $F'\supset F$, let us write $\GL_k(E\otimes_F F')\times {\bf G}_m(F')$ as a subgroup of ${\G}_n(F')$. As in \S\ref{sec:unitary-groups}, there are two possibilities, (i) and (ii), for $E\otimes_{F}F'$. In case (i), the group ${\bf G}_n(F')$ is given as in (\ref{gu}) 
and 
$$ \GL_k(E\otimes_{F}F')\times {\bf G}_m(F')
\ni (g,h)\mapsto m(g,h)
\in {\G}_n(F').$$
In case (ii), $E\otimes_F F'=F'\times F'$ with associated embeddings $\iota_1,\iota_2$ from $E\longrightarrow F'$ and ${\bf G}_n/_{F'}\cong \GL_{2n}\times \GL_1$ via $\iota_1$. Then ${\bf M}(F')$ as a subgroup of $\GL_{2n}(F')\times \GL_1(F')$ is given by 
\begin{equation}\label{embed}
 (g_1,g_2, h, \eta)\mapsto \left(\left[\begin{array}{c|c|c}\eta g_1&&\\\hline & h&\\\hline &&g_2^{\star}\end{array}\right],\eta\right)\in \GL_{2n}(F')\times \GL_1(F'),
\end{equation}
where $g_1,g_2\in \GL_{k}(F'), h\in \GL_{2m}(F')$, and $\eta\in F'^{\times}$.

\subsubsection{$L$-groups and the adjoint action}
\label{sec:Lgps}
For any connected reductive group ${\bf G}$ defined over $F$, let $^L{}G=\widehat{G}\rtimes W_F$ be its $L$-group \cite{B79} with $\widehat{G}$ being its connected component. Here, $W_F$ is the Weil group of $F$ which is a dense subgroup of $\Gamma_F$. Its action on $\widehat{G}$ is inherited from the action of $\Gamma_F$ on $\widehat{G}$. Let $\Phi_n$ denote the $n\times n$ matrix $(\Phi_n)_{ij}=(-1)^{i+1}\delta_{i, n-j+1}$. Since ${\bf G}_n/_{\overline{F}}=\GL_{2n}\times \GL_1$, we have $\widehat{G}_n=\GL_{2n}(\C)\times\GL_1(\C)$. We identify the maximal torus $\widehat{T}\subset \widehat{G_n}$ with $(\C^{\times})^{2n}\times\C^{\times}$ as in (\ref{max-torus}). Let $\widehat{B}\supset \widehat{T}$ be the corresponding Borel subgroup of upper triangular matrices. We use the standard $F$-splitting $\{\widehat{B},\widehat{T},\{E_{i,i+1}\}\}$ with root vectors $\{E_{i,i+1}, 1\leq i\leq 2n-1\}$ and form the $L$-group $
^L{}G_n=\GL_{2n}(\C)\times \GL_1(\C)\rtimes \Gamma_F.$
The Galois action factors through $\text{Gal}(E/F)$ and the non-trivial Galois automorphism $\theta$ sends $
(g,\lambda)\mapsto (\Phi^{-1}_{2n}{}^t g^{-1}\Phi_{2n},\lambda\det(g))$. The embedding $\widehat{M}=\GL_k(\C)\times\GL_k(\C)\times \GL_{2m}(\C)\times \GL_1(\C)\hookrightarrow \widehat{G_n}$
is given by (\ref{embed}). For $(g_1,g_2,h,\eta)\in \widehat{M}$, the non-trivial Galois action interchanges $g_1\leftrightarrow g_2$, sends $h\mapsto h^{\star}$ and $\eta\mapsto \eta\det(h).$

Suppose $^L{}N$ is the $L$-group of ${\bf N}$ as defined in \cite{B79}, let  $^L{}\frak{n}$ denote the complex Lie algebra of $^L{}N$, then 
\[
^L{}{\mathfrak n}=\left\{\left(\begin{array}{c|c|c}{\bf 0}&X&Y\\\hline {\bf 0}&{\bf 0}&Z\\\hline{\bf 0}&{\bf 0}&{\bf 0}\end{array}\right); X\in \text{M}_{k\times 2m}(\C), Y\in \text{M}_{k\times k}(\C), Z\in M_{2m\times k}(\C)\right\}.
\]
Let $r$ denote the adjoint action of $\widehat{M}\rtimes W_F$ on $^L{}\mathfrak{n}$ -- an element $(g_1,g_2,h,\eta)\in \widehat{M}$ acts by sending 
\[
X\mapsto \eta g_1X h^{-1}, Y\mapsto \eta g_1Yg_2^{-1}, Z\mapsto hZ(g_{2}^{\star})^{-1};
\]
and the Galois action is given according to its action on the roots. For $s\geq 1$, let $\rho_s$ denote the standard representation of ${\rm GL}_s(\C)$ and $1_{s}$ the trivial representation of ${\rm GL}_s(\C)$. Then $r=r_1\oplus r_2,$
where 
\begin{eqnarray}
r_1|_{\widehat{M}}&=&\rho_k\otimes 1_k\otimes \tilde{\rho}_{2m}\otimes \rho_{1}\oplus 1_k\otimes\rho_k\otimes\rho_{2m}\otimes 1_{1}\nonumber\\
r_2|_{\widehat{M}}&=& \rho_k\otimes\rho_k\otimes 1_{2m}\otimes\rho_1\nonumber.
\end{eqnarray}

Let ${\rm As}$ denote the $n^2$ dimensional representation of 
${}^L{}(R_{E/F}\GL_n)=\GL_n(\C)\times \GL_n(\C)\rtimes \Gamma_F$
given by 
\[
\mbox{As}(x,y,\gamma)=\left\{ \begin{array}{ll}
x\otimes y, & \mbox{if $\gamma$ restricted to $E$ is trivial}\\
y\otimes x, & \mbox{if $\gamma$ restricted to $E$ is not
trivial.}  
\end{array}
\right.
\]
Then as a representation of ${}^L{}(R_{E/F}\GL_n)\times \C^{\times}$, $r_2={\rm As}\otimes \rho_1$.  Hence $r_2$ is referred to as the {\it twisted Asai representation}. 

\subsubsection{Local factors}
We continue with the pair $({\bf G}={\bf G}_n,{\bf P})$, ${\bf P}={\bf M}{\bf N}$, as above. The character $\psi$ defines a non-degenerate character $\psi^{G}$ of $U$ via 
\[
n(u,v)\mapsto \psi((u_{1,2}+\bar{u}_{1,2})+\cdots+(u_{n-1,n}+\bar{u}_{n-1,n})+v_{n,n+1}),
\]
where the matrix $n(u,v)$ is as in (\ref{unip}). It is compatible with the Weyl group representative $w_{0}$ in the sense that $
\psi^{G}(w_{0}^{-1}uw_{0})=\psi^{G}(u), u\in U\cap M.$
By a {\it generic} representation of $M$ (cf. \cite[Section 3]{S88}), we mean generic with respect to the restriction of $\psi^{G}$ to $U\cap M$. Let $\sigma$ be an irreducible admissible generic representation of $\GL_k(E)$ and let $\tau$ be an irreducible admissible generic representation of $G_m$, put $\pi=\sigma\times \tau$ to denote the corresponding generic representation of $M$. The Langlands-Shahidi method \cite{S90} applied to the triple $(G,M,\pi)$ then gives rise to the local factors $L(s,\sigma\times\tau,r_i)$ and $\gamma(s,\sigma\times\tau,r_i,\psi)$, $i=1,2$. We write $L(s,\sigma\times\tau)$ and $\gamma(s,\sigma\times\tau,\psi)$ to denote the factors corresponding to the representation $r_1$. Similarly, we have the local factors $L^{U}(s,\sigma\times\tau')$ and $\gamma^{U}(s,\sigma\times\tau',\psi)$ associated with the triplet $(H,M\cap H,\sigma\times\tau')$ for any irreducible admissible generic representation $\tau'$ of $H_m$. (Here, $H=H_n$.)
\subsubsection{The unramified calculation}
\label{sec-unr}
Assume in this section that $\tau$, $\sigma$, and $E/F$ are all unramified. Write $\tau=\tau(\nu_1,\nu_2,\ldots,\nu_m,\nu_{0})$ as the spherical constituent of a  unramified principal series representation, where $\nu_1,\nu_2,\ldots,\nu_m$ are unramified characters of $E^{\times}$, and $\nu_{0}$ is an unramified character of $F^{\times}$. Its central character $\omega_{\tau}$ is given by $\omega_{\tau}=\nu_1\nu_2\cdots\nu_m(\nu_{0}\circ N_{E/F})$. Similarly, write $\sigma=\sigma(\mu_1,\mu_2,\ldots,\mu_k)$, where $\mu_1,\mu_2,\dots,\mu_k$, are unramified characters of $E^{\times}$. Let $\chi$ be the corresponding unramified character of the maximal torus ${T}\subset {M}$ so that $\pi$ is the spherical constituent of the unramified principal series representation of ${M}$ induced from $\chi$. Let $\widetilde{\rho}$ be half the sum of the absolute roots generating ${\bf N}$, and let $\widetilde{\alpha}'$  be $\frac{\widetilde{\rho}}{({\widetilde{\alpha}_i}^{\vee},\widetilde{\rho})}$, where $\widetilde{\alpha}_i$ is one of the two simple roots outside $\widetilde{\theta}$. Let $\alpha'$ denote the restriction of $\widetilde{\alpha}'$ to the maximal split torus ${\bf A}_{0}$. Then one has the general formula 
\[
L(s,\pi,r_i)=\prod_{\substack{\gamma\in \Sigma^{+}-\Sigma_{\theta}^{+}\\\langle \alpha',\gamma\rangle=i}}(1-\chi(\gamma^{\vee}(\varpi_F))q_F^{-n_{\gamma}s})^{-1}; i=1,2,
\]
where $\langle \alpha',\gamma \rangle=\frac{2(\widetilde{\gamma},\widetilde{\alpha}')}{(\widetilde{\gamma},\widetilde{\gamma})}$ with $\widetilde{\gamma}$ any absolute root restricting to $\gamma$ and $n_{\gamma}$ is the number of absolute roots restricting to $\gamma$. From this we obtain 
\begin{eqnarray}
L(s,\pi,r_1)&=&\prod_{i=1}^{k} \prod_{j=1}^{m}L_{E}(s,\bar{\omega}_{\tau}\mu_i\nu^{-1}_j)L_E(s,\bar{\omega}_{\tau}\mu_i\nu_j),\nonumber \\
L(s,\pi,r_2)&=& \prod_{i=1}^{k}L_{F}(s,\nu_{0}\mu_i|_{F^{\times}})\prod_{1\leq i<j\leq k}L_{E}(s,(\nu_{0}\circ N_{E/F})\mu_{i}\mu_{j}).\nonumber 
\end{eqnarray}
Here, for ${\mathfrak k}=E$ or $F$ and an unramified character $\eta$ of ${\mathfrak k}^{\times}$, $L_{\mathfrak k}(s,\eta)=(1-\eta(\varpi_{\mathfrak k})q_{\mathfrak k}^{-s})^{-1}$, and the character $\bar{\omega}_{\tau}$ is the Galois conjugate of $\omega_{\tau}$ given by $\bar{\omega}_{\tau}(x)=\omega_{\tau}(\bar{x}), x\in E^{\times}$. 

\subsection{Base change lifts from even unitary similitude groups in the generic case}
\label{GF for GU}
We recall the notion of base change lifts (or transfers) from similitude unitary groups to general linear groups. Let $K/k$ be a quadratic extension of number fields (or local fields). Consider the groups ${\G}={\bf G}_{n}$ and ${\bf H}={\bf H}_{n}$ defined with respect to $K/k$ as in \S\ref{sec:unitary-groups}. Put $\widetilde{{\bf G}}={\rm Res}_{K/k} ({\rm GL}_{2n} \times {\rm GL}_1)$ and $\widetilde{{\bf H}}$=${\rm Res}_{K/k} ({\rm GL}_{2n})$. The identification ${\G}_n\simeq_{K} {\rm GL}_{2n}\times {\rm GL}_{1}$ identifies ${\rm Res}_{K/k}{\G}_n$ with $\widetilde{\G}$ and similarly for the group $\widetilde{\bf H}$. We introduced the dual group ${}^L G$ in \S\ref{sec:Lgps}, the corresponding group for the unitary group ${\bf H}={\bf H}_n$ is given by ${}^L H={\rm GL}_{2n}(\C)\rtimes \Gamma_k$ with the same Galois action on ${\rm GL}_{2n}(\C)$. Suppose $H \subset G$ denote the subgroups consisting of the ${\mathbf A}_k$-points (or $k$-points) of the algebraic groups ${\mathbf H} \subset {\mathbf G}$, respectively. The natural projection map from 
\[
{}^L G\longrightarrow {}^L H,\,(g,\lambda)\rtimes \gamma\mapsto g\rtimes \gamma,
\]
corresponds to taking irreducible constituents of $\pi|_{H}$ for an irreducible admissible representation $\pi$ of $G$. We note that ${}^L\widetilde{G} = {\rm GL}_{2n}(\mathbb{C}) \times {\rm GL}_1(\mathbb{C}) \times {\rm GL}_{2n}(\mathbb{C}) \times {\rm GL}_1(\mathbb{C}) \rtimes \Gamma_k$ with the Galois action factoring through the Galois group ${\rm Gal}(K/k)$ and the non-trivial Galois automorphism sends $(g_1,\lambda_1)\times (g_2,\lambda_2)\mapsto (g_2,\lambda_2)\times (g_1,\lambda_1)$. Similarly, ${}^L\widetilde{H} = {\rm GL}_{2n}(\mathbb{C}) \times {\rm GL}_{2n}(\mathbb{C}) \rtimes \Gamma_k$ with the non-trivial action given by $(g_1,g_2)\mapsto (g_2,g_1)$. 

The embedding $\widehat{G}\hookrightarrow \widehat{\widetilde{G}}$ given by $(g,\lambda)\mapsto ((g,\lambda),\theta(g,\lambda))$ extends to a morphism of $L$-groups
\[
{\rm BC}: {}^L{G} \rightarrow {}^L\widetilde{{G}}
\]
called the stable base change map (cf. \cite[Section 2]{KK08}). There is a similarly defined map ${\rm BC}^{U}: {}^L{H} \rightarrow {}^L\widetilde{H}$. (see Section 3 in \cite{KK05} for more details).

\subsubsection{Local Lifts in the Generic Case}

We deduce the existence of local lifts for generic representations of similitude unitary groups from the corresponding results in \cite{KK05}. Let $E/F$ be a quadratic extension of $p$-adic fields, as defined in \S2.1. Define $G:={\mathbf G}_n(F)$ and $H:= {\mathbf H}_n(F)$. As observed earlier, $ZH \subset G$ is a subgroup of index 2, placing us in the framework of \cite[Section 6]{BX}, which establishes the following key points:

\begin{enumerate}
    \item For any $\pi \in \Pi(G)$, the restriction $\pi|_H$ is multiplicity-free.
    \item The restriction decomposes as 
    \[
    \pi|_H = \pi_1 \oplus \pi_1 \circ \mathrm{Ad}(g),
    \]
    where $g = \mathrm{diag}(aI_n, I_n) \in G$ with $a \notin N_{E/F}(E^\times)$.
\end{enumerate}

When $\pi$ is $\psi$-generic, we fix $\pi_1$ as the unique $\psi$-generic constituent of $\pi|_H$.

\begin{definition}
An irreducible admissible representation $\Pi \otimes \chi$ of $\tilde{{\bf G}}(F) \cong \mathrm{GL}_n(E) \times \mathrm{GL}_1(E)$ is called the \textbf{local base change lift} of $\pi$ if:
\begin{itemize}
    \item $\chi = \bar{\omega}_\pi$,
    \item The central character $\omega_\Pi$ of $\Pi$ satisfies 
    \[
    \omega_\Pi(z) = \omega_\pi(z / \bar{z}),
    \]
    \item The following identities hold for all irreducible generic representations $\sigma$ of $\mathrm{GL}_m(E)$:
    \begin{equation}
    \gamma^U(s, (\sigma \otimes \bar{\omega}_\pi) \times \pi_1, \psi) = \lambda(E/F, \psi)^{2mn} \gamma(s, (\sigma \otimes \chi) \times \Pi, \psi \circ \mathrm{Tr}_{E/F}),
    \end{equation}
    \begin{equation}
    L^U(s, (\sigma \otimes \bar{\omega}_\pi) \times \pi_1) = L(s, (\sigma \otimes \chi) \times \Pi).
    \end{equation}
\end{itemize}
Here, $\lambda(E/F,\psi)$ denotes Langlands' $\lambda$-function associated with the pair $(E/F,\psi)$.
\end{definition}

\begin{remark}\label{rmk_unr_lift}
Suppose $E/F$, $\pi$ and $\sigma$ are all unramified. Write $\pi=\pi(\nu_1,\nu_2,\ldots$ $,\nu_m,\nu_{0})$ as the spherical constituent of a unramified principal series representation and write $\sigma=\sigma(\mu_1,\mu_2,\ldots,\mu_k)$, where $\mu_1,\mu_2,\dots,\mu_k$, are unramified characters of $E^{\times}$, as in \S\ref{sec-unr}, Let $\Pi$ be the unique unramified quotient $\Pi(\nu_1,\ldots,\nu_n,\nu_n^{-\theta},\ldots,\nu_n^{-\theta} )$. It follows from the unramified calculation in \S\ref{sec-unr} that $\Pi\otimes\bar{\omega}_{\pi}$ is the local base change lift of $\pi$. 
\end{remark}

For ramified representations, we deduce the existence of local lifts from the corresponding results for unitary groups in \cite{KK05}. 
\begin{theorem}\label{locallift}
Assume $E/F$ is a quadratic extension of $p$-adic fields. Let $\pi$ be a supercuspidal generic representation of $G$ and let $\omega_{\pi}$ be its central character. Then it has a local lift which is unique and of the form $(\rho_1 \boxplus \cdots \boxplus \rho_k) \otimes \bar{\omega}_{\pi}$, where $\rho_i$'s are supercuspidal representations of ${\rm GL}_{n_i}(E)$ for all $i=1, \cdots, k$ satisfying $\rho_i \cong \rho_i^{\epsilon}$, $\rho_i \not\cong \rho_j$ for $i\not=j$, and $L(s, \rho_i, r_A \otimes \delta_{E/F})$ has a pole at $s=0$.
\end{theorem}
\begin{proof}
Let $\pi'$ be an irreducible admissible generic constituent of $\pi|H$ and let $\Pi$ be its local lift as in \cite[Theorem 8.4]{KK05}, it is of the form $\rho_1\boxplus\cdots\boxplus\rho_k$ with $\rho_i$'s as stated above. By construction, $\Pi$ is a local component of a global automorphic representation whose unramified components are local base change lifts as in Remark~\ref{rmk_unr_lift} (see loc.cit.) The multiplicity one theorem for idele class characters then implies that $\omega_{\Pi}$ is given by $\omega_{\Pi}(z)=\omega_{\pi}(z/\bar{z})$. We conclude $\Pi\otimes\bar{\omega}_{\pi}$ is a local lift of $\pi$. If $\Pi'\otimes\bar{\omega}_{\pi}$ were another candidate lift (generic) of $\pi$, then it follows that
\[
\gamma(s, (\sigma\otimes\bar{\omega}_{\pi})\times \Pi, \psi \circ \operatorname{Tr}_{E / F})=\gamma(s, (\sigma\otimes\bar{\omega}_{\pi})\times \Pi', \psi \circ \operatorname{Tr}_{E / F})
\]
for all irreducible generic representations $\sigma$ of $\GL_m(E)$. We now appeal to the local converse theorem for general linear groups \cite{H93} to conclude $\Pi\simeq \Pi'$.
\end{proof}

\subsubsection{Global lifts in the generic case}
\begin{definition}\label{sbcl}
Let $\pi=\otimes_v \pi_v$ be a globally generic irreducible cuspidal representation of $\mathbf{G}\left(\mathbb{A}_{k}\right)$. We say an automorphic representation $\Pi \otimes \chi =\otimes_v (\Pi_v \otimes \chi_v)$ of $\widetilde{{\bf G}}(\mathbb{A}_{k}) \cong$ $\mathrm{GL}_{2n}(\mathbb{A}_{K}) \times \mathrm{GL}_{1}(\mathbb{A}_{K}) $ is the strong base change lift of $\pi$ if for every $v, \Pi_v \otimes \chi_v$ is a local base change lift of $\pi_v$, in the above sense. It is said to be a weak base change lift if $\Pi_v\otimes\chi_v$ is a local lift of $\pi_v$ for almost all places $v$. 
\end{definition}

\begin{theorem}\label{Global BC lift:GU2}
Let $\pi$ be a globally $\psi$-generic cuspidal representation of $\G(\mathbb{A}_k)$ with central character $\omega_{\pi}$. There exists a strong base change lift of $\pi$ and is necessarily of the form 
\begin{equation}
{\rm BC}(\pi) = ({\sigma}_1 \boxplus \cdots \boxplus {\sigma}_k)\otimes \bar{\omega}_{\pi},
\end{equation}
where $\sigma_i$'s are unitary cuspidal representation of ${\rm GL}_{n_i}(\mathbb{A}_K)$ such that $\sigma_i \not\cong \sigma_j$ for $i \not= j$ and $L(s,{\sigma}_i, {\rm As}_{K/k} \otimes \delta_{K/k})$ has a simple pole at $s=1,1 \leq i \leq k$.
\end{theorem}
\begin{proof}

For $\pi'$ as above, let $\Pi'={\rm BC}^{U}(\pi')$ be the strong base change lift of $\pi'$ as in \cite[Theorem 8.10]{KK05}, it is of the form $\sigma_1\boxplus\cdots\boxplus\sigma_k$ with $\sigma_i's$ satisfying the stated condition. Put $\Pi=\Pi'\otimes\bar{\omega}_{\pi}$, then by construction $\Pi_v$ is a local base change lift of $\pi_v$ for every $v$. On the other hand, as shown in \cite[Lemma 5.4]{KK08}, any other strong base change lift $\Pi_1$ of $\pi$ is necessarily of the form $\Pi_1'\otimes\bar{\omega}_{\pi}$ with $\Pi'_{1,v}\simeq {\rm BC}^{U}(\pi'_v)$ for almost all $v$. This means $\Pi'_1$ is a weak base change lift of $\pi'$ (in the context of unitary groups). We can now appeal to \cite[Theorem 6.8]{KK05} to conclude that $\Pi'_1$ is a full induced representation from unitary cuspidal representations of general linear groups. Hence by multiplicity one theorem, $\Pi'_v\simeq \Pi'_{1,v}$ for all $v$.  
\end{proof}
\begin{remark}\label{equality with LS}
If $v$ is a place where the associated data are all unramified, the calculation in \S\ref{sec-unr} implies that 
\begin{eqnarray}
 \gamma^{U}(s, (\sigma_v\otimes\bar{\omega}_{\pi_v}) \times \pi'_v, \psi_v)&=&\gamma(s,\sigma_v\times\pi_v,\psi_v),\nonumber \\
L^U(s,( \sigma_v\otimes\bar{\omega}_{\pi_v}) \times \pi'_v)&=& L(s, \sigma_v\times\pi_v).\nonumber
\end{eqnarray}
This also holds at a split place $v$. (cf. \cite[Remark 6.1]{KK05}.) In this case, the representation $\pi_v$ is identified with $\tau_v\otimes \xi_v$ as a representation of ${\rm GL}_n(F_v)\times F_v^{\times}$. Under this identification, ${\rm BC}(\pi_v)\simeq (\tau_v\otimes\tilde{\tau}_v)\otimes\bar{\omega}_{\pi_v}$. If $w_1,w_2$, are the places lying above $v$, it follows from \cite[Section 2]{KK08} that the local component of $\omega_{\pi}$ at $w_1$ is $\omega_{\tau_v}\xi_v$ and at $w_2$ is $\xi_v$. For the remaining finite places $v$, where $\pi_v$ is ramified, one can establish the above equalities when $\pi_v$ is supercuspidal via the local-to-global principal as in \cite[Section 5]{S90}. The case of a general $\pi_v$ will then follow from the multiplicative property of these local factors. 
\end{remark}

\subsection{\bf Generic Arthur packet conjectures}\label{s:generic Arthur}

We begin with a review of the concept of $L$-packets and $A$-packets in representation theory of $p$-adic groups. For $F$ as in Section~\ref{sec:prelims}, let $W'_F$ be the group $W'_F = W_F \times  {\rm SL}_2(\mathbb{C})$, where as usual $W_F$ is the Weil group of $F$. Let $I_F\subset W_F$ be the inertia subgroup of $F$ and let us fix a geometric Frobenius element ${\rm Fr}\in W_F$ which satisfies ${\rm Fr}(x^q)=x$ for $x$ in the residue field of the maximal unramified extension of $F$. We write $\|\cdot\|: W_F\longrightarrow {\C}^{\times}$ to denote the unramified character of $W_F$ that takes the value $q^{-1}$ on ${\rm Fr}$. For a connected reductive group ${\bf G}$ defined over $F$, let $\Pi(G)$ denote the set of equivalence classes of irreducible representations of $G$, where $G:=\G(F)$. A Langlands parameter is a homomorphism $\varphi$ from $W'_F\longrightarrow {}^L{G}$ satisfying
\begin{itemize}
\item[(1)] $\varphi: {\rm SL}_2(\C)\longrightarrow \widehat{G}$ is algebraic,
\item[(2)] $\varphi$ is trivial on a open subgroup of $I_F$ and $\phi(Fr)$ is semisimple,
\item[(3)] $\varphi$ is compatible with the projection onto the Weil group $W_F$.
\end{itemize}
Two such parameters $\varphi$, $\varphi'$, are said to be equivalent if there is an element $g\in \widehat{G}$ such that $\varphi'=Ad(g)\circ \varphi$. Let $\Phi(G)$ be the set of equivalence classes of Langlands parameters. The local Langlands conjecture asserts that to each $\varphi$ in $\Phi(G)$ corresponds a finite set $\Pi_{\varphi}$ of equivalence class of irreducible admissible representations of the group $G$ called $L$-packets which partition the set $\Pi(G)$. A Langlands parameter $\varphi$ is said to be tempered if the image $\varphi(W_F)$ of $W_F$ is bounded in $\widehat{G}$. We write $\Phi_{\rm temp}(G) \subset \Phi(G)$ to denote the subset of equivalence classes of all tempered parameters. Let $\Pi_{\rm temp}(G)$ denote the subset of $\Pi(G)$ that consists of tempered representations, in the sense that their characters are tempered distributions on $G$. 

Let us now recall the notion of $A$-packets as introduced by Arthur. For a detailed account, see \cite{A89, A13}. An Arthur parameter $\psi$ is a continuous homomorphism 
\[
\psi: W'_F\times {\rm SL}_2(\C)\longrightarrow {}^LG
\]
such that the restriction of $\psi$ to $W'_F$ belongs to $\Phi_{\rm temp}(G)$ and the restriction of $\psi$ to the supplemental ${\rm SL}_2(\C)$ factor is algebraic. Let $\Psi(G)$ denote the set of all Arthur parameters subject to the same equivalence as above. Arthur also introduced a larger set denoted as $\Psi^{+}(G)$ whose elements are equivalence classes of parameters $\psi$ but without the requirement that $\psi$ be bounded upon restriction to $W'_F$. For each Arthur parameter $\psi$, there is an associated Langlands parameter $\phi_{\psi}$ given by the formula
\[
\varphi_{\psi}(w)=\psi(w, \begin{pmatrix} \|w\|^{\frac{1}{2}} & 0 \\ 0 & \|w\|^{-\frac{1}{2}} \end{pmatrix}), w\in W'_F,
\]
where $\|\cdot\|$ on $W'_F$ is the pullback of the above defined unramified character of $W_F$. 

Arthur conjectured \cite{A13} that for each $\psi$ in $\Psi(G)$, there exists an associated finite set $\Pi_{\psi}$ of irreducible admissible representations of $G$, referred to as the $A$-packet of $\psi$, satisfying several defining properties. Primarily, $A$-packets are constrained to comprise solely unitary representations and must encapsulate the associated $L$-packet, meaning $\Pi_{\varphi_{\psi}}(G)\subset \Pi_{\psi}(G)$. A pivotal aspect of the $A$-packet of $\psi$ involves the presence of a pairing between a naturally affiliated finite group $A_{\psi}$ and the set $\Pi_{\psi}(G)$, thereby facilitating a parameterization of $\Pi_{\psi}(G)$ in terms of irreducible characters of the group $A_{\psi}$. Contrary to $L$-packets, $A$-packets are neither disjoint nor do they partition the set of equivalence classes of irreducible unitary representations of $G$. 

When $\G$ is a quasi-split classical group--such as a symplectic, special orthogonal or unitary group--Arthur's conjecture, and consequently the local Langlands conjecture have been proven. Arthur established these results for symplectic and orthogonal groups \cite{A13}, while Mok extended them to quasi-split unitary groups \cite{M15}.

We now introduce the generic Arthur packet conjecture.
\begin{conjecture}\label{Weak G}
Let $\psi$ and $\varphi_{\psi}$ be as above and let $\Pi(\varphi_{\psi})$ be the $L$-packet associated to $\varphi_{\psi}$. Assume $\Pi(\varphi_{\psi})$ has a generic member with respect to some Whittaker datum. Then $\varphi_{\psi}$ is tempered.
\end{conjecture}

\begin{remark}
Assuming the equality of the Langlands-Shahidi and Artin L-functions, the third author proved Conjecture \ref{Weak G} for any (quasi-split) ${\bf G}$ (cf. \cite[Theorem 5.1]{S11}). For additional progress on Conjecture \ref{Weak G} in the context of classical groups and ${\rm GSpin}$ groups, see the works of Ban, Jantzen-Liu, Liu, Kim, and Heiermann-Kim \cite{Ban06, JL14, L11, HeK17, K21}.
\end{remark}

\begin{definition}\label{def:t}
An $L$-packet $\Pi_{\varphi}(G)$ is said to be tempered if all its members are tempered representations.
\end{definition}

Here is another version of the above conjecture: 
\begin{conjecture}\label{Strong G}
Let $\psi, \varphi_{\psi},$ and $\Pi(\varphi_{\psi})$ be as in Conjecture \ref{Weak G}. If $\Pi(\varphi_{\psi})$ has a generic member with respect to any Whittaker datum, then it is a tempered $L$-packet.
\end{conjecture}

This paper aims to unconditionally prove these conjectures for unitary and similitude unitary groups of even rank.

\section{\bf The local conjectures for unitary similitude groups}

For a quadratic extension $E/F$ of non-archimedean local fields, we set $G={\bf G}_n(F)$ and $H={\bf H}_n(F)$ as before. There is a natural relationship between $L$-parameters of $G$ and those of $H$ as discussed in Section \ref{sec:prelims}. 

Namely, recall the projection maps from the group ${}^LG$ into ${}^L H$ and into ${}^L(R_{E/F}{\rm GL}_1)$. On the connected component of ${}^L G$ these maps are given by 
\[
(A,y)\mapsto A\,\ \ {\rm and } \ \ (A,y)\mapsto ((\det A)y,y),
\]
respectively. Any parameter $\phi: W_F'\longrightarrow {}^L G$ gives rise to a pair $(\psi,\chi)$ via these projections:
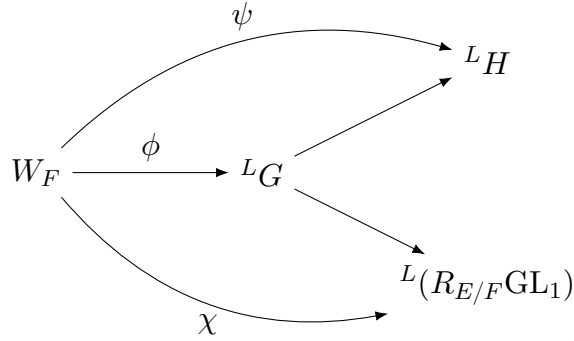
\begin{center}
\begin{tikzpicture}[>=Latex, every node/.style={scale=1.1}]

\node (X) at (0,0) {$W_F$};
\node (Y) at (3,0) {${}^L G$};
\node (Z) at (6,1.5) {${}^L H$};
\node (W) at (6,-1.5) {${}^L(R_{E/F}\mathrm{GL}_1)$};

\draw[->] (X) -- node[above] {$\phi$} (Y);
\draw[->] (Y) -- (Z); 
\draw[->] (Y) -- (W); 

\draw[->, bend left=30] (X) to node[above] {$\psi$} (Z);
\draw[->, bend right=30] (X) to node[below] {$\chi$} (W);

\end{tikzpicture}
\end{center}
Note that $\chi$ is equivalent to a character of $E^{\times}$ via class field theory, by abuse of notation, we also write it as $\chi$. In the other direction, given parameters $(\psi,\chi)$, we define the corresponding $\phi: W_F\longrightarrow {}^L G$ as follows: On $W_E$ it is given by 
\[
x\mapsto (\psi(x),\chi(\bar{z})), x\in W_E,
\]
where $x\mapsto z$ is under the class field theory map. For a fixed representative $w_c\in W_F/W_E$, suppose $\psi(w_c)=a\rtimes \gamma$ and $\chi(w_c)=(x,y)\rtimes \gamma$, then $\phi(w_c)=(a,z)\rtimes \gamma$ where $z$ is such that $\det(a)z^2=xy$.

\subsection{Construction of generic $L$-packets}
\label{Sec5.2}

\begin{theorem}\label{Langlands parameter}
Let $\pi$ be an irreducible admissible generic representation of $G$. Then, there exists a Langlands parameter $\phi_{\pi}: W_F'\longrightarrow {}^LG$ such that for any irreducible admissible generic representation $\rho$ of ${\rm GL}_k(F)$, $k\geq 1$, we have the following equalities:
$$L(s, \rho \times \pi) = L(s, \phi_{\rho} \otimes \phi_{\pi}) \textit{\ \ and \ \ } \gamma(s, \rho \times \pi, \psi_F) = \gamma(s, \phi_{\rho} \otimes \phi_{\pi}, \psi_F),$$
where $\phi_{\rho}$ is the Langlands parameter associated to $\rho$ through the local Langlands correspondence for general linear groups \cite{HT01, H00}.
\end{theorem}

\begin{proof}
Let us first assume $\pi$ is supercuspidal. In that case, if $\pi'$ is the irreducible admissible generic constituent of $\pi|H$, it is also supercuspidal.
Let $\Pi'$ be the local lift of $\pi'$ to ${\rm GL}_{2n}(E)$ as in the proof of Theorem~\ref{locallift}. Then, it is of the form $\Pi'=\rho_1 \boxplus \cdots \boxplus \rho_k$, where the $\rho_i$'s are mutually non-isomorphic conjugate self-dual supercuspidal representations of ${\rm GL}_{n_i}(E)$ and satisfying $L(s, \rho_i, r_A \otimes \delta_{E/F})$ has a pole at $s=0$, $1\leq i\leq k$. The Langlands parameter $\phi_{\Pi'}: W_E'\longrightarrow {\GL}_{2n}(\C)$ associated to $\Pi'$ is given by $\phi_{\rho_1} \oplus \cdots \oplus \phi_{\rho_k}$, where $\phi_{\rho_i}$ is the Langlands parameter associated to $\rho_i$ through the local Langlands correspondence for ${\rm GL}_{n_i}(E)$. In the terminology of \cite[\S~2.2]{M15}, the parameter $\phi_{\Pi'}$ is conjugate self-dual with parity $-1$. Let $\phi_{\pi'}: W_F' \rightarrow ^L H$ be the parameter associated to $\phi_{\Pi'}$ via \cite[Lemma 2.2.1]{M15}. If one takes the character $\chi_\kappa$ in loc. cit. to be trivial, then the associated map $\xi_{\chi_\kappa, *}$ from $\Phi(H)$ to $\Phi({\rm GL}_{2n}(E))$ as defined by Mok coincides with our base change map $\mathrm{BC}^U$. In particular, the restriction of $\phi_{\pi'}$ to $W'_E$ recovers $\phi_{\Pi'}$, i.e.,
\[
\phi_{\pi'}|_{W'_E}=\phi_{\Pi'}.
\]
Now, let $\phi_{\pi}: W_F' \rightarrow \ ^L{G}$ be the parameter associated with the pair $(\phi_{\pi'},\bar{\omega}_{\pi})$ in the sense described above. It follows from Theorem \ref{locallift} and Remark \ref{equality with LS} imply that we have
\[
L(s, \rho \times \pi) = L^U(s,( \rho \otimes\bar{\omega}_{\pi}) \times \pi') = L(s, \phi_{\rho} \otimes ({\rm BC}\circ \phi_{\pi})) 
\]
\[
\gamma(s, \rho \times \pi, \psi_F) =  \gamma^{U}(s, (\rho \otimes \bar{\omega}_{\pi}) \times \pi', \psi_F) = \gamma(s, \phi_{\rho} \otimes ({\rm BC}\circ\phi_{\pi}), \psi_F).
\]

For non-supercuspidal admissible representation, we follow the construction of Heiermann \cite{H06a}. Namely, first consider the discrete series case $\pi=\pi_{\rm ds}$. 
There exist a parabolic subgroup $P=MN$ of $G$, an irreducible supercuspidal representation $\pi_{\rm sc}$ of $M$, and a point $\lambda$ in the real vector space $a_M^*=X(M)_F\otimes\R$ such that $\pi_{\rm ds}$ is a subquotient of $\mathrm{Ind}_{P}^{G}\left(\pi_{\rm sc} \otimes \chi_\lambda\right)$. (Here, $X(M)_F$ is the group of rational characters of $M$ and $\chi_\lambda$ is the usual unramified quasi-character associated to the point $\lambda$). If $\phi_{\pi_{\rm sc}}$ is the Langlands parameter attached to $\pi_{\rm sc}$ as above, then the Langlands parameter $\phi_{\pi_{\rm ds}}$ is given by 
\[
\phi_{\pi_{ds}}(\gamma, h)=\phi_{\pi_{sc}}(\gamma, 1) \phi_{\pi_{sc}, \lambda}(h) \textit{ for }\gamma \in W_F \textit{ and } h \in {\rm SL}_2(\mathbb{C}) 
\]
as explained in \cite[Section 6.2]{H06a}. The notation $\phi_{\pi_{sc}, \lambda}(h)$ is explained in loc.cit.

To pass from discrete series to an irreducible admissible representation of $G$, we use Langlands classification. In particular, there is a parabolic subgroup $P_1=M_1N_1$ of $G$, a tempered representation $\pi_t$ of $M_1$ and a point $\lambda_{M_1}$ in $a_{M_1}^*$ in the positive Weyl chamber of $P_1$ such that $\pi$ is the Langlands quotient of the induced representation $\mathrm{Ind}_{P_1}^{G}\left(\pi_t \otimes \chi_{\lambda_{M_1}}\right)$. Furthermore, there is a parabolic subgroup $P_2=M_2N_2$ of ${\bf M}_1$ and a discrete series $\pi_{\rm ds}$ such that $\pi_t \hookrightarrow Ind^{M_1}_{P_2\cap M_1}(\pi_{\rm ds})$. Note that we constructed the Langlands parameter $\phi_{\pi_{ds}}$ that corresponds to $\pi_{\rm ds}$. Then in \cite[Section 6.2]{H06a} the Langlands parameter $\psi_\pi: W_F \times \mathrm{SL}_2(\mathbb{C}) \rightarrow{ }^L G$ is constructed as follows:
\[
(\gamma, h) \mapsto s_{M_1}^{v_F(\gamma)} \phi_{\pi_{ds}}(\gamma, h),
\]
where $s_{M_1}$ is semisimple element in $C({ }^L M_1)^{\circ}$ (connected component of center of ${ }^L M_1$) that corresponds to $\lambda_{M_1}$ through the local Langlands correspondence for torus. Note that we have the equalities of local factors \cite[Theorem 4.4]{H06b}.
\end{proof}

Suppose $\pi$ is an irreducible admissible generic representation of $G$ and $\phi=\phi_{\pi}$ as above. Recall we have the conjectural set $\Pi_{\phi}(G)$, the following result characterizes all generic members in this set in terms of $L$-functions from the Langlands-Shahidi method.
\begin{theorem}\label{generic reps in $L$-packet}
Let $\pi_1$ (resp. $\pi_2$) be an irreducible admissible generic representation of $G$ and let $\phi_{\pi_1}$ (resp. $\phi_{\pi_2}$) be its Langlands parameter as constructed in Theorem \ref{Langlands parameter}. Then, $\phi_{\pi_1} \cong \phi_{\pi_2}$ if and only if for each irreducible admissible generic representation $\rho$ of ${\rm GL}_r(F)$ we have the following:
$$L(s, \rho \times \pi_1) = L(s, \rho \times \pi_2) \textit{\ and \ } \gamma(s, \rho \times \pi_1, \psi_F) = \gamma(s, \rho \times \pi_2, \psi_F). $$ 
\end{theorem}
\begin{proof}
Assume that $\phi_{\pi_1} \cong \phi_{\pi_2}$. Theorem \ref{Langlands parameter} directly tells us the equalities of local factors. 
We now assume the equality of local factors. We have $L(s, \rho \times \Pi_1) = L(s, \rho \times \Pi_2)$ for any irreducible admissible generic representation $\rho$ of ${\rm GL}_r(F)$, where $\Pi_1$ (resp. $\Pi_2$) is local functorial lift of $\pi_1$ (resp. $\pi_2$). Then, local converse theorem for general linear groups implies that $\Pi_1 \cong \Pi_2$ \cite{H93}. Note that construction of Langlands parameter in the proof of Theorem \ref{Langlands parameter} only depends on its local functorial lift. Therefore, $\phi_{\pi_1} \cong \phi_{\pi_2}$.
\end{proof}

This points to a way of defining the generic $L$-packet $\Pi_{\phi}(G)$ under the following assumption.

\begin{assumption}\label{Assumption:LS}
Assume the existence of local factors for non-generic representations $\pi$ that respects Langlands classification.  
\end{assumption}

\begin{definition}[Generic $L$-packets]\label{Def:generic L-packet}
For any fixed irreducible admissible generic representation $\pi_0$ of $G$ with $L$-parameter $\phi=\phi_{\pi_0}$, let $\Pi_{\phi}(G)$ be the set of equivalence classes of irreducible admissible representations $\pi$ such that for each irreducible admissible representation $\rho$ of ${\rm GL}_r(F)$ we have the following:
$$L(s, \rho \times \pi) = L(s, \rho \times \pi_0 ) \textit{\ and \ } \gamma(s, \rho \times \pi, \psi_F ) = \gamma(s, \rho \times \pi_0, \psi_F). $$ 

\end{definition}

\begin{remark}
    One main purpose of the paper is to remove Assumption \ref{Assumption:LS} so that applications of our main theorems, i.e., Theorems \ref{StrongConjG}, \ref{StrongConjG:U} become completely unconditional. We remove it in Section \ref{Removing A}.
\end{remark}
We prove the following property for generic $L$-packets in the case of even unitary similitude groups: 

\begin{lemma}\label{lem:tempered $L$-packet}
If an L-packet $\Pi_{\phi}$ for $G$ contains a tempered representation that is also generic, then all other members in $\Pi_{\phi}$ are tempered as well (i.e., $\Pi_{\phi}$ is a tempered $L$-packet).
\end{lemma}

\begin{proof}
Let $\pi_0$ be an irreducible tempered generic representation of $G_{n_0}$ in $\Pi_{\phi}$ and suppose that there exists a non-tempered representation $\pi \in \Pi_{\phi}$. Since $\pi$ and $\pi_0$ are in the same $L$-packet, we have the equality of $L$-functions $L(s, \rho \times \pi) = L(s, \rho \times \pi_0)$ for any irreducible admissible representation $\rho$ of ${\rm GL}_s(F)$, $s \geq 1$. 

The Langlands classification implies that $\pi$ can be considered to be the unique quotient of the following induced representation:
$$\nu^{r_1} \tau_1 \times \cdots \times \nu^{r_m} \tau_m \rtimes\pi_t,$$
where $\tau_i$ is a tempered representation of $\GL_{n_i}(F)$, $n_i \geq 1$ for each $i=1, \cdots, m$ with $r_1 > r_2 > \cdots > r_m>0$, and $\pi_t$ is a tempered representation of $G_{n_t}$ ($n_t$ and $n_0$ are in the same parity). Since $\pi$ is non-tempered, $m \geq 1$.

Since $L$-function follows the Langlands classification (Assumption \ref{Assumption:LS}), we have
\[
L(s, \rho \times \pi_0) = L(s, \rho \times \pi) = L(s, \rho \times \pi_t) \displaystyle\prod\limits_{i=1}^{m} L(s+r_i, \rho \times \bar{\omega}_{\pi_0}\tau_i)L(s-r_i, \rho \times \bar{\omega}_{\pi_0}\widetilde{\widebar{\tau_i}} )
\]
for any admissible representation $\rho$ of ${\rm GL}_s(F)$, $s \geq 1$.
If we let $\rho$ be a contragredient representation of $\bar{\omega}_{\pi_0}\widetilde{\widebar{\tau_i'}}$, then $L(s - r_j, \rho \times \bar{\omega}_{\pi_0}\widetilde{\widebar{\tau_i'}})$ has a pole at $s= r_j>0$. This also implies that $L(s, \rho \times \pi_0)=L(s, \rho \times \pi)$ has a pole at $s=r_j>0$ when  $\rho$ is a contragredient representation of $\bar{\omega}_{\pi_0}\widetilde{\widebar{\tau_i'}}$. This contradicts that $L(s, \rho \times \pi_0)$ is holomorphic for $Re(s) >0$ since both $\rho$ and $\pi_0$ are tempered and generic \cite{HeO13}; therefore, $m=0$. We conclude that all members in $\Pi_{\phi}$ are tempered.
\end{proof}

\subsubsection{A weak version of the generic Arthur packet conjectures for even unitary similitude groups}
\label{Sec5.1}

Thanks to \cite[Theorem 5.1]{S11}, our theorem \ref{Langlands parameter} implies Conjecture \ref{Weak G} in the case of even unitary similitude groups. 

\begin{theorem}\label{wConjG:GU}
Let $\psi$ be an Arthur parameter for $G$ and $\phi_{\psi}$ be its corresponding Langlands parameter. Assume that there exists an $L$-packet $\Pi(\phi_{\psi})$ that corresponds to the Langlands parameter $\phi_{\psi}$. If $\Pi(\phi_{\psi})$ has a generic member, then $\phi_{\psi}$ is a tempered Langlands parameter.
\end{theorem}

We now prove the strong version of the generic Arthur packet conjecture (Conjecture \ref{Strong G}) for even unitary similitude groups. We first study $L$-packets through the local Langlands correspondence. Then we further prove several properties of $L$-packets and Langlands parameters in the following subsections.

\subsubsection{A strong version of the generic Arthur packet conjecture for even unitary similitude groups}
\label{Sec5.3}
We first prove the embedding properties of discrete series of ${\bf G}_n(F)$.

\begin{proposition}\label{GU:embedding}
Let $\pi_{ds}$ be an irreducible discrete series generic representation of $G$. Then $\pi_{ds}$ can be embedded into the following induced representation:
\begin{align}\label{DS:SC:generic}
\pi_{ds} \hookrightarrow \delta_1 \times \cdots \times \delta_r \times \delta_1' \times \cdots \times \delta_k' \rtimes \pi_{cusp}.
\end{align}
where $\delta_i = \delta([\nu^{a_i} \rho_i, \nu^{b_i} \rho_i])$, $a_i \leq 0, a_i + b_i > 0$, $\rho_i$ is an irreducible unitary supercuspidal representation of ${\rm GL}_{n_i}(F)$, $n_i \geq 1$ for $i=1, \ldots, r$, $\delta_j' =  \delta([\nu^{a_{\rho_j'}} \rho_j' ,\nu^{b^{(j)}} \rho_j']))$ are as in \cite[(4.5)]{KM19} (i.e., $\delta_j'$ are from classification of strongly positive representations), $\pi_{cusp}$ is an irreducible supercuspidal generic representation of $G_{n_c}$, $n_c \geq 1$. Furthermore, each $a_{\rho_j'}$ is in $\{ \frac{1}{2}, 1 \}$ for $j = 1, \ldots, k$.
\end{proposition}
\begin{proof}
The proof is similar to \cite[Section 6]{K15} and we briefly explain the main ideas of the proof instead of repeating similar arguments. In the case when $\pi_{ds}$ is strongly positive, the theorem follows with $r=0$. Suppose that $\pi_{ds}$ is not strongly positive. In the same way as in the proof of \cite[Theorem 5.3]{K15}, we conclude that there exist a sequence of segments $\Delta_1, \ldots, \Delta_k$ satisfying $e\left(\Delta_1\right) \leq \cdots \leq e\left(\Delta_k\right)$ and an irreducible supercuspidal representation $\pi_{cusp} \in R(G_{n_c})$, $n_c \geq 1$, such that we have $\pi_{ds} \hookrightarrow \delta\left(\Delta_1\right) \times \cdots \times \delta\left(\Delta_k\right) \rtimes \pi_{cusp}$. Write $\Delta_i:=\left[\nu^{a_i} \rho_i, v^{b_i} \rho_i\right]$, where $\rho_i \in R(GL)$ is an irreducible unitary supercuspidal representation for $i=1, \ldots, k$. Let $a:=\min \left\{a_i \mid 1 \leq i \leq k\right\}$ and let $j$ be such that $a_j=a$. Then minimality of $a_j$ implies that there exists an irreducible representation $\pi_1 \in R(G_{n_1})$, $n_1 \geq 1$, such that $\pi_{ds} \hookrightarrow \delta\left(\Delta_j\right) \rtimes \pi_1$. We show that $\pi_1$ is a discrete series representation. Suppose that $\pi_1$ is not a discrete series representation. Then there exists an embedding of the form $\pi_1 \hookrightarrow \delta\left(\left[v^{a^{\prime}} \rho, v^{b^{\prime}} \rho\right]\right) \rtimes \pi^{\prime}$, where $a^{\prime}+b^{\prime} \leq 0$. Then, minimality of $a$ again implies that we have $\pi_{ds} \hookrightarrow$ $\delta\left(\left[v^{a^{\prime}} \rho, v^{b^{\prime}} \rho\right]\right) \times \delta\left(\left[v^{a_j} \rho_j, v^{b_j} \rho_j\right]\right) \rtimes \pi^{\prime}$ which is a contradiction since $\pi$ is a discrete series representation. We conclude that $\pi_1$ is also a discrete series representation. If $\pi_1$ is strongly positive, the theorem follows with $r=1$. If not, in the same way as above, $\pi_1$ can be embedded into the representation $\delta\left(\left[v^{a_{j^{\prime}}} \rho_{j^{\prime}}, v^{b_{j^{\prime}}} \rho_{j^{\prime}}\right]\right) \rtimes \pi_2$, where $a_{j^{\prime}} \leq 0, a_{j^{\prime}}+b_{j^{\prime}}>$ 0 and $\pi_2$ is a discrete series representation. We repeat this argument until we get strongly positive. Then, the theorem follows.
\end{proof}

\begin{lemma}\label{TT}
Let $\pi$ be an irreducible admissible generic representation of $G$. Then, $\pi$ is a tempered representation if and only if its local functorial lift $\Pi$ is a tempered representation of ${\rm GL}_s(F)$, $s \geq 1$. 
\end{lemma}

\begin{proof}
Suppose that $\pi$ is a non-tempered admissible generic representation of $G$ and its local functorial lift $\Pi$ is a tempered representation of ${\rm GL}_s(F)$, $s \geq 1$. Then, due to \cite{HeM07}, there exists $r_1 > r_2 > \cdots > r_m>0$ such that $\pi$ can be written as the following induced representation:
\[
\nu^{r_1} \tau_1 \times \cdots \times \nu^{r_m} \tau_m \rtimes \pi_t,
\]
where $\tau_i$ is a tempered representation of ${\rm GL}_{n_i}(F)$, $n_i \geq 1$ for each $i=1, \cdots, m,$ and $\pi_t$ is a tempered generic representation of $G_{n_t}$ for some $n_t < n$. Since $\pi$ is non-tempered, $m \geq 1$.

Now we consider the $L$-function $L(s, \omega_{\pi}\widetilde{\widebar{\tau_j}} \times \pi)$. Since the $L$-functions from the Langlands-Shahidi method are defined by the Langlands classification, for any admissible representation $\rho$ we have the following:
$$L(s, \rho \times \pi) = L(s, \rho \times \pi_t) \displaystyle\prod\limits_{i=1}^{m} L(s+r_i, \rho \times \bar{\omega}_{\pi}\tau_i)L(s-r_i, \rho \times \bar{\omega}_{\pi}\widetilde{\widebar{\tau_i}}).$$
As in the proof of Lemma \ref{lem:tempered $L$-packet}, we conclude that $L(s, \rho \times \Pi)=L(s, \rho \times \pi)$ has a pole at $s=r_j>0$ when $\rho$ is a contragredient representation of $\bar{\omega}_{\pi}\widetilde{\widebar{\tau_j}}$. This contradicts that $L(s, \rho \times \Pi)$ is holomorphic for $Re(s) >0$ since both $\rho$ and $\Pi$ are tempered and generic \cite{HeO13}; therefore, $\pi$ is also a tempered representation. 

For the sufficient condition we exactly follow the arguments in \cite[Proposition 8.5]{KK05} and we skip the proof and briefly mention which tools that we need to apply those arguments. We use Theorem \ref{locallift} and Theorem \ref{GU:embedding} to show the local functorial lift of discrete series of $G$ is tempered. Then it is straightforward to show that the functorial lift of tempered representations are also tempered. 
\end{proof}

Now we are ready to prove a strong version of the generic Arthur packet conjecture in the case of even unitary similitude groups under Assumption \ref{Assumption:LS}. Let $\psi, \phi_{\psi}$ be as in Conjecture \ref{Weak G}.
Let $\Pi(\phi_{\psi})$ be a generic $L$-packet that corresponds to $\phi_{\psi}$ defined in Definition \ref{Def:generic L-packet}. This means that we assume that $\Pi(\phi_{\psi})$ contains a generic member.

\begin{theorem}\label{StrongConjG}
Let $\psi \in \Psi(G)$ be an Arthur parameter for even unitary similitude groups and $\phi_{\psi}$ be its corresponding Langlands parameter. Suppose the $L$-packet $\Pi(\phi_{\psi})$ in Definition \ref{Def:generic L-packet} that corresponds to the $L$-parameter $\phi_{\psi}$ has a generic member. Then $\Pi(\phi_{\psi})$ is a tempered $L$-packet.
\end{theorem}

\begin{proof}
Let $\pi$ be a generic member in $\Pi(\phi_{\psi})$. 
Theorem \ref{wConjG:GU} implies that the corresponding $L$-parameter $\phi_{\psi}: W_F' \rightarrow \,^{L}G_n = \widehat{G}_n \rtimes W_F'$ is a tempered $L$-parameter; therefore, the image $\phi_{\psi}(W_F)$ in $\widehat{G}_n$ is bounded. 
Through the stable base change map, we have 
\[
{\rm BC}: \,^{L}{G}_n \hookrightarrow \,^{L}({\rm Res}_{E/F}({\rm GL}_{2n} \times {\rm GL}_1)).
\]
Then, the image ${\rm BC} \circ \phi_{\psi}(W_F)$ in $\widehat{{\rm GL}_{2n} \times {\rm GL}_1}$ is also bounded; therefore, ${\rm BC} \circ \phi_{\psi}$ becomes a tempered $L$-parameter for ${\rm GL}_{2n} \times {\rm GL}_1$. According to the local Langlands correspondence for $GL$ groups \cite{HT01, H00}, there exists an irreducible tempered representation $\Pi$ of ${\rm GL}_{2n}(E) \times {\rm GL}_1(E)$ that corresponds to ${\rm BC} \circ \phi_{\psi}$ such that
\begin{align}\label{E:GL}
\gamma(s, \rho \times \Pi, \psi_F) = \gamma(s, \phi_{\rho} \otimes ({\rm BC} \circ \phi_{\psi}), \psi_F) \textit{\ \ and \ \ } L(s, \rho \times \Pi) = L(s, \phi_{\rho} \otimes ({\rm BC} \circ \phi_{\psi}))
\end{align}
for any irreducible admissible representation $\rho$ of ${\rm GL}_s(F)$, $s \geq 1$ where $\phi_{\rho}$ is the $L$-parameter that corresponds to $\rho$ through the local Langlands correspondence for ${\rm GL}$ groups \cite{HT01, H00}. Here, the local factors on the left-hand side are Rankin-Selberg local factors \cite{JPSS83}, and those on the right-hand side are Artin factors. The equalities \eqref{E:GL} also imply the following:
\[
\gamma(s, \rho \times \pi, \psi_F) = \gamma(s, \phi_{\rho} \otimes \phi_{\psi}, \psi_F) = \gamma(s, \phi_{\rho} \otimes ({\rm BC} \circ \phi_{\psi}), \psi_F)= \gamma(s, \rho \times \Pi, \psi_F)
\]
\[
\text{and }L(s, \rho \times \pi) = L(s, \phi_{\rho} \otimes \phi_{\psi}) = L(s, \phi_{\rho} \otimes ({\rm BC} \circ \phi_{\psi}))= L(s, \rho \times \Pi).
\]
Therefore, $\Pi$ is a local functorial lift of $\pi$.

Since $\Pi$ is tempered, Lemma \ref{TT} implies that $\pi$ is a tempered representation as well. Now, it remains to prove that all other members in $\Pi(\phi_{\psi})$ are tempered as well. Since $\Pi(\phi_{\psi})$ contains a tempered generic representation $\pi$, Lemma \ref{lem:tempered $L$-packet} implies that $\Pi(\phi_{\psi})$ is a tempered $L$-packet.
\end{proof}

We now briefly explain how we can adapt the above idea to prove the generic Arthur packet conjecture in the case of even unitary groups as well.

\begin{theorem}\label{StrongConjG:U}
Let $\psi \in \Psi(H)$ be an Arthur parameter for even unitary groups and $\phi_{\psi}$ be its corresponding Langlands parameter. Suppose the $L$-packet $\Pi(\phi_{\psi})$ that corresponds to the $L$-parameter $\phi_{\psi}$ has a generic member. Then $\Pi(\phi_{\psi})$ is a tempered $L$-packet.    
\end{theorem}
\begin{proof}
The first step, i.e., equality of $L$-functions from the Langlands-Shahidi method and Artin $L$-functions is proved in \cite[Proposition 8.7]{KK05}. Therefore, a weak version of the generic Arthur packet conjecture for $H$ is true due to \cite[Theorem 5.1]{S11}. It now remains to prove a strong version of the conjecture. Let us point out three important steps of the proof. First, the construction of Langlands parameter for $H$ is already described in the proof of Theorem \ref{Langlands parameter}. Second, we define $L$-packets exactly as in Definition \ref{Def:generic L-packet} with the same assumption \ref{Assumption:LS}. Third, Lemmas \ref{lem:tempered $L$-packet} and \ref{TT} can be generalized to even unitary groups since it depends on the properties of $L$-functions from the Langlands-Shahidi methods and classification of discrete series, which are available for $H$. Therefore, we can apply the argument in the proof of Theorem \ref{StrongConjG} to the case of even unitary groups and this completes the proof.
\end{proof}

\section{Removing Assumption \ref{Assumption:LS}}\label{Removing A}
The purpose of this section is to make Theorems \ref{StrongConjG} and \ref{StrongConjG:U} unconditional. Namely, we define generic $L$-packets for ${\bf G}_n$ so that we can remove Assumption \ref{Assumption:LS} for $G:={\bf G}_n(F)$ and $H:={\bf H}_n(F)$. 

\subsection{$L$-packets for $H$}
Let us recall $L$-packets for $H$ defined in \cite[Theorem 2.5.1(b)]{M15}. Briefly, for an Arthur parameter $\psi \in \Psi(H)$, the Arthur packet $\Pi_{\psi}$ is defined in \cite[Theorem 2.5.1(a)]{M15}. This includes definitions of $L$-packets that correspond to tempered Langlands parameter since $\psi$ becomes a tempered Langlands parameter when it is trivial on the second $SL_2(\mathbb{C})$-factor. Furthermore, $L$-packets, in general, are defined as follows: 
Recall in \cite{A82} that a member $\phi$ of $\Phi(H)$ determines a commuting pair $\phi_t \in \Phi_{temp}(H)$ and a $\phi_{+} \in \Phi(H)$ that satisfies the following:
\begin{equation}\label{L-parameter_decomp}
\phi(w)=\phi_t(w) \phi_{+}(w) \quad\left(w \in W_F'\right),
\end{equation}
The centralizer of the image of $\phi_{+}$ in ${ }^L H$ is a Levi subgroup ${ }^L  M_{H}$. Note that $\phi_{+}$ will be a map into ${ }^L A_{H}$, where ${\bf A}_{H_n}$ is the split component of ${\bf M_{H_n}}$ 
Note also that the image of $\phi_t$ is in ${ }^L M_{H}$ since $\phi_t$ and $\phi_{+}$ commute. Let $\Pi_{\phi_t}$ be a tempered $L$-packet that corresponds to $\phi_t$ through \cite[Theorem 2.5.1(b)]{M15}. Let $\nu: A_{H} \rightarrow \mathbb{C}^*$ be a character that corresponds to $\phi_{+}$. For $\pi_t \in \Pi_{\phi_t}$, let $I(\nu, \pi_t)=\operatorname{Ind}_{P}^{G} (\pi_t \otimes \nu)$. We assume that this is a standard module after the conjugation of Weyl group elements. Then, $L$-packets for $H$ in general are defined as follows:
\begin{definition}\label{def_Lpacket}
Let $\phi, \phi_t, \phi_{+}, \Pi_{\phi_t}$, and $\nu$ be as above. $L$-packet $\Pi_{\phi}$ is defined as follows:
\[
\Pi_{\phi}=\{J(\nu, \pi_t) \mid \pi_t \in \Pi_{\phi_t}\}
\]
where $J(\nu, \pi_t)$ is unique Langlands quotient of a standard module $I(\nu, \pi_t)$.
\end{definition}

\subsection{Base change lifts and $L$-packets for $G$}
We describe global stable base change lifts for ${\bf G}_n(\mathbb{A}_k)$, where $k$ is a number field. This is discussed in \cite[Lemma 5.4]{KK08} for globally generic cuspidal representations. With Mok's work \cite{M15} in hand, the proof of this lemma may now be extended to any cuspidal automorphic representation. Namely, we have

\begin{theorem}\label{sbc:non-generic:GU}
Let $\underline{\pi}$ be an irreducible cuspidal representation of ${\bf G}_n(\mathbb{A}_k)$ and let $\omega_{\underline{\pi}}$ be its central character. By \cite[Proposition 1.8.1]{HL}, the restriction $\underline{\pi}|_{\bf H(\mathbb{A}_k)}$ is a direct sum of cuspidal representations of ${\bf H}(\mathbb{A}_k)$. Let $\underline{\pi}^{\prime}$ be an irreducible cuspidal constituent of $\underline{\pi}|_{\bf H(\mathbb{A}_k)}$. Then the stable base change lift of $\underline{\pi}$ is given by ${\rm BC}^{U}(\underline{\pi}^{\prime}) \otimes \bar{\omega}_{\underline{\pi}}$, where ${\rm BC}^{U}(\underline{\pi}^{\prime})$ is the stable base change lift of $\underline{\pi}^{\prime}$ as established in \cite{M15}. It is denoted by ${\rm BC}(\underline{\pi})$. 
\end{theorem}
\begin{proof}
We follow exactly as in the arguments of \cite[Lemma 5.4]{KK08} and we briefly mention how it can be generalized to non-generic representations. Recall in Subsection \ref{GF for GU} that ${\rm BC}: {}^L G \rightarrow {}^L\widetilde{
G}$ (resp. ${\rm BC}^{U}:$ ${}^LH \rightarrow {}^L\widetilde{H}$) is the stable base change map for even unitary similitude groups (resp. even unitary groups). Then ${\rm BC}^{U}(\underline{\pi}^{\prime})$ is the Langlands functorial lift of $\underline{\pi}^{\prime}$ through the stable base change map ${\rm BC}^{U}$ \cite{M15}. We have the following diagram:

\[
\begin{tikzcd}
{}^LG \arrow[r] \arrow[d, "{\rm BC}"]
& {}^LH \arrow[d, "{\rm BC}"] \\
{}^L\widetilde{G} \arrow[r]
& {}^L\widetilde{H}
\end{tikzcd}
\]
where the horizontal arrows are natural projections (over the Galois group) onto the corresponding $\GL_{2n}$ factor. We now with this setup follow exactly in the proof of \cite[Lemma 5.4]{KK08} and we conclude that the stable base change lift of $\underline{\pi}$ is given by ${\rm BC}^{U}(\underline{\pi}^{\prime}) \otimes \bar{\omega}_{\underline{\pi}}$.
\end{proof}

We now consider the local case. Let $\pi$ be an irreducible admissible unitarizable representation of $G$ and $\omega_{\pi}$ be its central character. First, \cite[Theorem, Appendix 1] {H83} implies that there exist a number field $k$ and an irreducible cuspidal automorphic representation $\underline{\pi}$ of ${\bf G}_n(\mathbb{A}_k)$ such that $\underline{\pi}_v = \pi$. Let $\underline{\pi}^{\prime}$ be an irreducible cuspidal constituent of $\underline{\pi}|_{\bf H(\mathbb{A}_k)}$. Theorem \ref{sbc:non-generic:GU} implies that ${\rm BC}(\underline{\pi}) = {\rm BC}^{U}(\underline{\pi}^{\prime}) \otimes \bar{\omega}_{\underline{\pi}}$. Therefore, ${\rm BC}(\pi) = {\rm BC}^{U}(\pi^{\prime}) \otimes \bar{\omega}_{\pi}$, where $\pi^{\prime} = \underline{\pi}^{\prime}_v$.   
Therefore, following Theorem \ref{sbc:non-generic:GU}, we now naturally define the local base change lifts for $G$ which is independent of the existence of global stable base change lift:

\begin{definition}\label{def_BC}
Let $\pi$ be an irreducible admissible unitarizable representation of $G$ and $\omega_{\pi}$ be its central character. Let $\pi'$ be an irreducible constituent of $\pi | H$. Then the stable base change lift ${\rm BC}(\pi)$ of $\pi$ is definend as
${\rm BC}^{U}(\pi') \otimes \bar{\omega}_{\pi}$.
\end{definition}

Accordingly, We also define $L$-packets using the local base change lifts.
\begin{definition}\label{Def_L-packets}
Let $\pi_0$ be an irreducible admissible representation of $G$. Then we define a local $L$-packet that contains $\pi_0$ as a set of irreducible admissible representations $\pi$ of $G$ such that 
\begin{enumerate}
    \item $\omega_{\pi} = \omega_{\pi_0},$
    \item ${\rm BC}(\pi) \cong {\rm BC}(\pi_0).$ 
\end{enumerate}
\end{definition}

Furthermore, we can define local factors using the definition of local base change lift as follows:

\begin{definition}\label{Lfunction_GU}
Let $\pi$ be an irreducible admissible unitarizable representation of $G$. Then for any irreducible admissible representation $\sigma$ of $\GL(E)$ we define local factors as follows:
\[
\begin{aligned}
& \gamma(s, \sigma \times \pi, \psi)=\lambda(E / F, \psi)^{2mn} \gamma(s, \sigma \times {\rm BC}(\pi), \psi \circ \operatorname{Tr}_{E / F}), \\
& L(s, \sigma \times \pi)=L(s, \sigma \times {\rm BC}(\pi)),
\end{aligned}
\]
Here, local factors in the R.H.S. are Rankin-Selberg local factors or the ones from the Langlands-Shahidi method \cite{JPSS83, S84}.
\end{definition}

We now prove several important properties of $L$-packets such as Shahidi's conjecture. 

\begin{proposition}\label{property_Lpacket_GU}

\begin{enumerate}
    \item An $L$-packet for $ G$ consists of finite members.
    \item A tempered $L$-packet contains a generic representation.
    \item $L$-funtions follow the Langlands classification.
    \item $L$-functions defined in Definition \ref{Lfunction_GU} agree with $L$-functions from the Langlands-Shahidi method in the generic case. 
\end{enumerate}
\end{proposition}
\begin{proof}
We first prove (1), i.e., finiteness of $L$-packets for $G$. Let $\pi$ be an irreducible admissible unitarizable representation of $G$ and $\Pi$ be an $L$-packet that contains $\pi$.
Let $\pi_1$ be one member in $\Pi$ and let $\pi'$ (resp. $\pi_1'$) be an irreducible constituent in $\pi|_{H}$ (resp. $\pi_1|_{H}$) as in Definition \ref{def_BC}. Then, Theorem \ref{sbc:non-generic:GU} and Definition \ref{Def_L-packets} imply that ${\rm BC}^{U}(\pi') \otimes \bar{\omega}_{\pi} \cong {\rm BC}(\pi) \cong {\rm BC}(\pi_1) \cong {\rm BC}^{U}(\pi_1') \otimes \bar{\omega}_{\pi_1} \cong {\rm BC}^{U}(\pi_1') \otimes \bar{\omega}_{\pi}$ since the central characters of the members in $\Pi$ are same. Therefore, for any irreducible admissible representation $\sigma$ of ${\rm GL}(F)$ we have
\[
\begin{aligned}
& \gamma(s, \sigma \times {\rm BC}^{U}(\pi_1')\otimes \bar{\omega}_{\pi}, \psi)=\gamma(s, \sigma \times {\rm BC}^{U}(\pi')\otimes \bar{\omega}_{\pi}, \psi), \\
& L(s, \sigma \times {\rm BC}^{U}(\pi_1')\otimes \bar{\omega}_{\pi})=L(s, \sigma \times {\rm BC}^{U}(\pi')\otimes \bar{\omega}_{\pi}),
\end{aligned}
\]
This implies that $\pi_1'$ and $\pi'$ are in the same $L$-packet for $H$. This completely describes all members in $\Pi$. This implies the finiteness of $\Pi$ due to the following two facts; first, $L$-packets for $H$ is finite \cite{M15}, and second, there exist finite irreducible constituents in $\pi|_{H}$ \cite{GK82}. Therefore, an $L$-packet for $G$ consists of finite members.\\

We now prove (2), i.e., Shahidi's conjecture for $G$. Suppose $\Pi$ is a tempered $L$-packet for $G$. For $\pi_t \in \Pi$, take an irreducible component $\pi_t'\subset \pi_t|_{H}$ in the restriction. Let $\Pi'$ be the corresponding tempered $L$-packet for $H$. Since Shahidi's conjecture for the group $H$ is known to be true, we conclude $\Pi'$ contains a generic member $\pi_g'$. Both $\pi_g'$ and $\pi_t'$ have the same central character given by the restriction of $\omega_{\pi_t}$ to the center $Z\cap H$ of $H$. Thus the pair $(\pi_g',\omega_{\pi_t})$ defines a representation of the subgroup $ZH\subset G$. Since this is a subgroup of finite index, it follows from a standard argument using Frobenius reciprocity that there is an irreducible generic representation $\pi_g$ of $G$ with central character $\omega_{\pi_t}$ containing $\pi_g'$ upon restriction to $H$.  Since ${\rm BC}^U(\pi_g') \cong {\rm BC}^U(\pi_t')$, it follows that ${\rm BC}(\pi_g) \cong {\rm BC}(\pi_g') \otimes \bar{\omega}_{\pi_g} \cong {\rm BC}(\pi_t') \otimes \bar{\omega}_{\pi_t} \cong {\rm BC}(\pi_t)$. We conclude $\pi_g$ is a member of the $L$-packet $\Pi$.

We now prove (3). Let $\pi$ be a non-tempered representation of $G$ and consider it the Langlands quotient of the induced representation 
\[
\nu^{r_1} \tau_1 \times \cdots \times \nu^{r_m} \tau_m \rtimes \pi_t,
\]
where $\tau_i$ is an irreducible tempered representation of ${\rm GL}(F)$ for each $i=1, \cdots, m,$ and $\pi_t$ is an irreducible tempered representation of $G_{n_t}$ for some $n_t < n$. 

The direct calculation implies that there exists a non-tempered representation $\pi'$ of $H$ that appears in $\pi | H$ such that it is the Langlands quotient of $\nu^{r_1} \tau_1 \times \cdots \times \nu^{r_m} \tau_m \rtimes \pi'_t$ for some irreducible tempered representation  $\pi'_t$ of $H_{n_t}$ that appears in $\pi_t | H_{n_t}$ since $\left(\nu^{r_1} \tau_1 \times \cdots \times \nu^{r_m} \tau_m \rtimes \pi_t\right) |_{H}$ is $\nu^{r_1} \tau_1 \times \cdots \times \nu^{r_m} \tau_m \rtimes \left(\pi_t|_{H_{n_t}}\right)$. Then, due to Definitions \ref{def_BC}, \ref{Lfunction_GU} with Theorem \ref{sbc:non-generic:GU} and property that $L$-functions for $H$ follows the Langlands classification, we have the following:
\[
\begin{aligned}
L(s, \sigma \times \pi)
&
=L(s, \sigma \times {\rm BC}(\pi)) = L(s, \sigma \times {\rm BC}^{U}(\pi')\otimes \bar{\omega}_{\pi})\\
&
= L(s, \pi_t' \otimes \bar{\omega}_{\pi} \times \sigma ) \prod_{i=1}^m L(s+r_i, \tau_i \otimes \bar{\omega}_{\pi} \times \sigma)  
L(s-r_i, \bar{\widetilde{\tau_i}} \otimes \bar{\omega}_{\pi} \times \sigma)
\end{aligned}
\]

Therefore, this implies that  $L$-functions for $G$ defined in Definition \ref{Lfunction_GU} also follows the Langlands classification.  \\
For (4), this is due to proof of Theorem \ref{Langlands parameter} and our construction of {\rm BC}. See also Remark \ref{rem_L} 
\end{proof}
\begin{remark}\label{rem_L}
Note that Definition \ref{Lfunction_GU} gives a new definition of $L$-functions for $G$ in general including non-generic representations. In the generic case this is compatible with Theorem \ref{Langlands parameter} in chapter 3, i.e., $L$-functions from the Langlands-Shahidi method. Furthermore, note that in the tempered case, Proposition \ref{property_Lpacket_GU} implies that a tempered $L$-packet contains a generic one and therefore we can also extend the definition of $L$-functions from the Langlands-Shahidi method from the generic case to tempered case. Also using Langlands classification we can extend $L$-functions from the Langlands-Shahidi method in full generality. This is compatible with our definition (Definition \ref{Lfunction_GU}) due to Proposition \ref{property_Lpacket_GU}.
\end{remark}

We conclude that we define $L$-functions that follows the Langlands classification.
\begin{corollary}
    Assumption \ref{Assumption:LS} is true.
\end{corollary}

\end{document}